\def\benm{\begin{enumerate}}
\def\eenm{\end{enumerate}}

\def\bal{\begin{align}}
\def\eal{\end{align}}

\documentclass[11pt]{amsart}
\usepackage{amsmath, amsthm, amscd, amsfonts,amssymb, graphicx}
\usepackage{color}
\newtheorem{theorem}{Theorem}[section]

\theoremstyle{definition}

\newtheorem{proposition}[theorem]{Proposition}
\newtheorem{corollary}[theorem]{Corollary}

\theoremstyle{remark}
\newtheorem{remark}[theorem]{Remark}

\numberwithin{equation}{section}

\newcommand{\dd}{\mathrm{d}}

\begin{document}

\title[Covariant Functions of Characters of Normal Subgroups]{Harmonic Analysis of Covariant Functions of Characters of Normal Subgroups}

\author[A. Ghaani Farashahi]{Arash Ghaani Farashahi}
\address{Department of Pure Mathematics, School of Mathematics, University of Leeds, Leeds LS2 9JT, United Kingdom}
\email{a.ghaanifarashahi@leeds.ac.uk}
\email{ghaanifarashahi@outlook.com}

\curraddr{}




\subjclass[2010]{Primary 43A15, 43A20, 43A85.}



\keywords{covariance property, normal subgroup, covariant function, character.}
\thanks{E-mail addresses: a.ghaanifarashahi@leeds.ac.uk (Arash Ghaani Farashahi) }

\begin{abstract}
Let $G$ be a locally compact group with the group algebra $L^1(G)$ and $N$ be a closed normal subgroup of $G$. Suppose that $\xi:N\to\mathbb{T}$ is a continuous character and $L_\xi^1(G,N)$ is the $L^1$-space of all covariant functions of $\xi$ on $G$. We showed that $L^1_\xi(G,N)$ is isometrically isomorphic to a quotient space of $L^1(G)$. It is also proved that the dual space $L^1_\xi(G,N)^*$ is isometrically isomorphic to $L^\infty_\xi(G,N)$. 
\end{abstract}

\maketitle

\section{\bf{Introduction}}
The covariant functions of characters (one-dimensional continuous irreducible unitary representations) of closed subgroups arise as building blocks with different applications in variant areas such as number theory (automorphic forms), induced representations, homogeneous spaces,  complex (hypercomplex) analysis, coherent states, and covariant analysis, see \cite{RB, FollP, kan.taylor, kisil.book, kisil1, MackII, MackI, pere}. 

In general, classical harmonic analysis methods cannot be applied for covariant functions of characters of arbitrary closed subgroups. In the case of a compact subgroup, harmonic analysis on covariant functions studied in \cite{AGHF.cov.com}. The following paper presents a unified operator theoretic approach to study abstract harmonic analysis on covariant functions of characters of closed normal subgroups. We consider $L^1$-spaces of covariant functions of characters of normal subgroups in locally compact groups. It is discussed that the introduced approach extends techniques used in \cite{AGHF.cov.com} when the subgroup is normal and compact. The presented approach also generalizes classical methods of abstract harmonic analysis and functional analysis of $L^1$-spaces on quotient (factor) groups by considering the character as the trivial character of the normal subgroup, see \cite{Bour, Brac, der1983, 50}. 

This article contains 5 sections and organized as follows. Section  2  is  devoted  to  fix  notations  and  provides  a  summary
of classical harmonic analysis on locally compact compact groups and factor (quotient) groups of locally compact groups. 
Let $G$ be a locally compact group with the group algebra $L^1(G)$ 
and $N$ be a closed normal subgroup of $G$. Suppose that $\xi:N\to\mathbb{T}$ is a fixed character of $N$. In section 3 we study abstract harmonic analysis of covariant functions of the character $\xi$ using a classical operator theoretic approach. 
Next section considers the Banach space $L^1_\xi(G,N)$, the $L^1$-space of covariant functions of $\xi$ on $G$, and investigates abstract harmonic analysis on $L^1_\xi(G,N)$. We then study some fundamental  properties of the Banach space $L^1_\xi(G,N)$. Section 5 is devoted to show that $L^1_\xi(G,N)$ is isometrically isomorphic to a quotient space of $L^1(G)$. We conclude the paper by showing that the dual space $L^1_\xi(G,N)^*$ is isometrically isomorphic to $L^\infty_\xi(G,N)$.  

\section{\bf{Preliminaries and Notations}}

Let $X$ be a locally compact Hausdorff space. Then $\mathcal{C}_c(X)$ denotes the space of all continuous complex valued functions on $X$ with compact support. 
If $\lambda$ is a positive Radon measure on $X$, for each $1\le p<\infty$ 
the Banach space of equivalence classes of $\lambda$-measurable complex valued functions $f:X\to\mathbb{C}$ such that
$$\|f\|_{L^p(X,\lambda)}:=\left(\int_X|f(x)|^p\dd\lambda(x)\right)^{1/p}<\infty,$$
is denoted by $L^p(X,\lambda)$ which contains $\mathcal{C}_c(X)$ as a $\|\cdot\|_{L^p(X,\lambda)}$-dense subspace.

Let $G$ be a locally compact group. A left Haar measure on $G$ is a non-zero positive Radon measure $\lambda_G$ on $G$ which satisfies $\lambda_G(xE)=\lambda_G(E)$ for every Borel
subset $E\subseteq G$ and every $x\in G$. It is well-known that every locally compact group $G$ possesses a left Haar measure (Theorem 2.10 of \cite{FollH}) and left Haar measures on $G$ are unique up to scaling, see Theorem 2.20 of \cite{FollH}. If $\lambda_G$ is a fixed left Haar measure on $G$ and $1\le p<\infty$, then $L^p(G)$ stands for the Banach function space $L^p(G,\lambda_G)$. For a function $f:G\to\mathbb{C}$ and $x\in G$, the functions $L_xf,R_xf:G\to\mathbb{C}$ are given by $L_xf(y):=f(x^{-1}y)$ and $R_xf(y):=f(yx)$ for $y\in G$.
The modular function of $G$, denoted by $\Delta_G$ is the continuous homomorphism $\Delta_G:G\to(0,\infty)$ which satisfies  
\[
\int_Gf(y)\dd\lambda_G(y)=\Delta_G(x)\int_GR_xf(y)\dd\lambda_G(y),
\]
for every $x\in G$, left Haar measure $\lambda_G$ on $G$ and $f\in L^1(G)$.

Suppose that $Aut(G)$ is the group of all bicontinuous group automorphisms of $G$ and $\alpha\in Aut(G)$. 
The Haar modulus $\sigma_G(\alpha)\in(0,\infty)$ is given by  
\begin{equation}\label{Haar.Modulu}
\int_G f(y)\dd\lambda_G(y)=\sigma_G(\alpha)\int_G f(\alpha(y))\dd\lambda_G(y),
\end{equation}
for every left Haar measure $\lambda_G$ on $G$ and $f\in L^1(G)$, see \cite{HR1}. Then $\sigma_G:Aut(G)\to(0,\infty)$ given by $\alpha\mapsto\sigma_G(\alpha)$ is a group homomorphism. It should be mentioned that existence of the positive real number $\sigma_G(\alpha)$ satisfying (\ref{Haar.Modulu}) guaranteed by the uniqueness of left Haar measures on the locally compact group $G$.

Let $N$ be a closed normal subgroup of $G$ with the left Haar measure $\lambda_N$. Then each left coset of $N$ is a right coset of $N$ and hence the coset space $G/N$ is a locally compact group called as factor (quotient) group of $N$ in $G$. Then $\mathcal{C}_c(G/N)$ consists of all functions $T_N(f)$, where 
$f\in\mathcal{C}_c(G)$ and
\begin{equation}\label{5.1}
T_N(f)(xN)=\int_Nf(xs)\dd\lambda_{N}(s)\quad (xN\in G/N).
\end{equation}
The factor (quotient) group $G/N$ has a left Haar measure $\lambda_{G/N}$, which is normalized with respect to the following Weil's formula  
\begin{equation}\label{6}
\int_{G/N}T_N(f)(xN)\dd\lambda_{G/N}(xN)=\int_Gf(x)\dd\lambda_G(x),
\end{equation}
for all $f\in L^1(G)$, see Thereom 3.4.6 of \cite{50}.

A character $\xi$ of $N$, is a continuous group homomorphism $\xi:N\to\mathbb{T}$, where $\mathbb{T}:=\{z\in\mathbb{C}:|z|=1\}$ is the circle group. In terms of group representation theory, each character of $N$ is a 1-dimensional irreducible continuous unitary representation of $N$. We then denote the set of all characters of $N$ by $\chi(N)$. A function $\psi:G\to\mathbb{C}$ satisfies covariant property associated to the character $\xi\in\chi(N)$, if $
\psi(xs)=\xi(s)\psi(x)$, for $x\in G$ and $s\in N$. In this case, $\psi$ is called as a covariant function of $\xi$.

\section{\bf Covariant Functions of Characters}
Throughout this section, we shall present some of the theoretical aspects of continuous covariant functions of characters of closed normal subgroup. In this direction, we first review some fundamental properties of continuous covariant functions. The covariant functions developed in abstract harmonic analysis and group representation theory in the construction of induced representations, see \cite{FollH, kan.taylor}. We here discuss one of the classical approaches to produce covariant functions.

Suppose that $\lambda_N$ is a left Haar measure on $N$.
For $\xi\in\chi(N)$ and $f\in\mathcal{C}_c(G)$, define the function $T_\xi(f):G\to\mathbb{C}$ via 
\[
T_\xi(f)(x):=\int_Nf(xs)\overline{\xi(s)}\dd\lambda_N(s),
\]
for every $x\in G$. 

Suppose $\mathcal{C}_\xi(G,N)$ is the linear subspace of $\mathcal{C}(G)$ given by 
\[
\mathcal{C}_\xi(G,N):=\{\psi\in\mathcal{C}_c(G|N):\psi(xk)=\xi(k)\psi(x),\ {\rm for\ all}\ x\in G,\ k\in N\},
\]
where 
\[
\mathcal{C}_c(G|N):=\{\psi\in\mathcal{C}(G):\mathrm{q}({\rm supp}(\psi))\ {\rm is\ compact\ in}\ G/N\},
\]
and $\mathrm{q}:G\to G/N$ is the canonical map given by $\mathrm{q}(x):=xN$ for $x\in G$. It is proven that the linear operator $T_\xi$ maps $\mathcal{C}_c(G)$ onto $\mathcal{C}_\xi(G,N)$, see Proposition 6.1 of \cite{FollH}.

\begin{remark}\label{J1}
If $\xi=1$ is the trivial  character of $N$, we then have $T_1(f)=T_N(f)$.
In this case, $\mathcal{C}_1(G,N)$ consists of functions on $G$ which are constant on cosets of $N$. Therefore, $\mathcal{C}_1(G,N)$ can be canonically identified with $\mathcal{C}_c(G/N)$ via $\psi\mapsto\widetilde{\psi}$, where $\widetilde{\psi}:G/N\to\mathbb{C}$ is given by $\widetilde{\psi}(xN):=\psi(x)$ for every $x\in G$. 
\end{remark}

We then have the following observations.

\begin{proposition}\label{J.xi.RkLy}
{\it Let $G$ be a locally compact group, $N$ be a closed subgroup of $G$, and $\xi\in\chi(N)$. Suppose $k\in N$ and $y\in G$. Then, 
\begin{enumerate}
\item $T_\xi\circ R_k=\Delta_N(k^{-1})\xi(k)T_\xi$ on $\mathcal{C}_c(G)$.
\item $T_\xi\circ L_y=L_y\circ T_\xi$ on $\mathcal{C}_c(G)$.
\end{enumerate}
}\end{proposition}
\begin{proof}
(1) Suppose that $f\in\mathcal{C}_c(G)$, and $k\in N$. Let $x\in G$ be given. We then have  
\begin{align*}
T_\xi(R_kf)(x)&=\int_Nf(xsk)\overline{\xi(s)}\dd\lambda_N(s)=\int_Nf(xs)\overline{\xi(sk^{-1})}\dd\lambda_N(sk^{-1})
\\&=\Delta_N(k^{-1})\xi(k)\int_Nf(xs)\overline{\xi(s)}\dd\lambda_N(s)=\Delta_N(k^{-1})\xi(k)T_\xi(f)(x),
\end{align*}
which implies that $T_\xi(R_kf)=\Delta_N(k^{-1})\xi(k)T_\xi(f)$.\\
(2) Suppose that $f\in\mathcal{C}_c(G)$ and $y\in G$. Let $x\in G$ be given. We then have
\begin{align*}
T_\xi(L_yf)(x)=\int_NL_yf(xs)\overline{\xi(s)}\dd\lambda_N(s)=\int_Nf(y^{-1}xs)\overline{\xi(s)}\dd\lambda_N(s)=L_y(T_\xi(f))(x),
\end{align*}
implying that $T_\xi(L_yf)=L_y(T_\xi(f))$.
\end{proof}

\subsection{Covariant Functions of Characters of Normal Subgroups}
We then continue by investigation of some aspects of covariant functions of characters of normal subgroups on locally compact groups. Throughout, let $G$ be a locally compact group and $N$ be a closed normal subgroup of $G$. Then there exists a unique homomorphism $\sigma_N:G\to (0,\infty)$ such that 
\begin{equation}\label{tau.main}
\int_Nv(s)\dd\lambda_N(x^{-1}sx)=\sigma_N(x)\int_Nv(s)\dd\lambda_N(s),
\end{equation}
and 
\[
\int_Nv(s)\dd\mu_N(x^{-1}sx)=\sigma_N(x)\int_Nv(s)\dd\mu_N(s),
\]
for all left Haar measure $\lambda_N$ of $N$, right Haar measure $\mu_N$ of $N$, $v\in\mathcal{C}_c(N)$, and $x\in G$. It is worthwile to mention that existance of the homomorphism $\sigma_N:G\to (0,\infty)$ guaranteed by the uniqueness of left Haar measures on the locally compact group $N$, Theorem 2.20 of \cite{FollH}. We then have $\sigma_N(t)=\Delta_N(t)$ for $t\in N$ and hence $\sigma_G(x)=\Delta_G(x)$ for $x\in G$. Also, the modular function of $G/N$ satisfies $\Delta_G(x)=\sigma_N(x)\Delta_{G/N}(xN)$ for $x\in G$, see Proposition 11 of \cite[Chap. VII, \S2]{Bour}. In particular, if $N$ is central we then have $\sigma_N=1$.

We then study some properties of $T_\xi$, when the subgroup $N$ is normal in $G$. 

\begin{proposition}\label{J.xi.normal}
{\it Let $G$ be a locally compact group and $N$ be a closed normal subgroup of $G$. Suppose $\xi\in\chi(N)$, $x\in G$, and $f\in\mathcal{C}_c(G)$. We then have  
\[
T_\xi(f)(x)=\sigma_N(x)\int_Nf(sx)\overline{\xi(x^{-1}sx)}\dd\lambda_N(s).
\]
In particular, if $N$ is a central subgroup of $G$ then 
\[
T_\xi(f)(x)=\int_Nf(sx)\overline{\xi(s)}\dd\lambda_N(s).
\]
}\end{proposition}
\begin{proof}
Let $f\in \mathcal{C}_c(G)$ be given. Then, for $x\in G$, we have  
\begin{align*}
T_\xi(f)(x)
&=\int_Nf(xs)\overline{\xi(s)}\dd\lambda_N(s)
=\int_Nf(xsx^{-1}x)\overline{\xi(s)}\dd\lambda_N(s)
\\&=\int_Nf(sx)\overline{\xi(x^{-1}sx)}\dd\lambda_N(x^{-1}sx)
=\sigma_N(x)\int_Nf(sx)\overline{\xi(x^{-1}sx)}\dd\lambda_N(s).
\end{align*}
If $N$ is central in $G$ then $\sigma_N(x)=1$ and $\xi(x^{-1}sx)=\xi(s)$ for all $x\in G$ and $s\in N$. Thus, we get 
\begin{align*}
T_\xi(f)(x)
&=\sigma_N(x)\int_Nf(sx)\overline{\xi(x^{-1}sx)}\dd\lambda_N(s)
=\int_Nf(sx)\overline{\xi(s)}\dd\lambda_N(s).
\end{align*}
\end{proof}

\begin{proposition}
{\it Let $G$ be a locally compact group, $N$ be a closed normal subgroup of $G$, and $\xi\in\chi(N)$. Suppose $f\in\mathcal{C}_c(G)$, $x\in G$, and $k\in N$. We then have  
\[
T_\xi(L_kf)(x)=\overline{\xi(x^{-1}kx)}T_\xi(f)(x).
\]
}\end{proposition}
\begin{proof}
Let $\tau_x(k):=x^{-1}kx$. Then, using Proposition \ref{J.xi.RkLy}(2), we have 
\begin{align*}
T_\xi(L_kf)(x)&=L_kT_\xi(f)(x)=T_\xi(f)(k^{-1}x)
\\&=T_\xi(f)(x\tau_{x}(k^{-1}))=\xi(\tau_{x}(k^{-1})T_\xi(f)(x)=\overline{\xi(\tau_x(k))}T_\xi(f)(x).
\end{align*}
\end{proof}

For functions $f\in\mathcal{C}_c(G)$ and $\psi\in \mathcal{C}_\xi(G,N)$, the functions $f\overline{\psi}$ and $\psi\overline{f}$ are continuous with compact supports. Therefore, they are integrable over $G$ with respect to the left Haar measure $\lambda_G$. We shall denote the later integrals by $\langle f,\psi\rangle$ and $\langle\psi,f\rangle$ respectively. 

\begin{theorem}
{\it Let $G$ be a locally compact group and $N$ be a closed normal subgroup of $G$. 
Suppose $\xi\in\chi(N)$ and $f,g\in\mathcal{C}_c(G)$. We then have 
\begin{equation}\label{Jxi.Jxi*}
\langle T_\xi(f),g\rangle=\langle f,T_\xi(g)\rangle.
\end{equation}
}\end{theorem}
\begin{proof}
Let $f,g\in\mathcal{C}_c(G)$ be given. We then have 
\begin{align*}
\langle T_\xi(f),g\rangle&=\int_GT_\xi(f)(x)\overline{g(x)}\dd\lambda_G(x)
\\&=\int_G\left(\int_Nf(xs)\overline{\xi(s)}\dd\lambda_N(s)\right)\overline{g(x)}\dd\lambda_G(x)\\&=\int_N\left(\int_Gf(xs)\overline{g(x)}\dd\lambda_G(x)\right)\overline{\xi(s)}\dd\lambda_N(s)
\\&=\int_N\left(\int_Gf(x)\overline{g(xs^{-1})}\dd\lambda_G(xs^{-1})\right)\overline{\xi(s)}\dd\lambda_N(s)
\\&=\int_N\Delta_G(s^{-1})\left(\int_Gf(x)\overline{g(xs^{-1})}\dd\lambda_G(x)\right)\overline{\xi(s)}\dd\lambda_N(s).
\end{align*}
Using normality of $N$ in $G$ and Proposition 3.3.17 of \cite{50}, we get $\Delta_G(s)=\Delta_N(s)$ for every $s\in N$. Therefore, we obtain  
\begin{align*}
\langle T_\xi(f),g\rangle
&=\int_N\Delta_N(s^{-1})\left(\int_Gf(x)\overline{g(xs^{-1})}\dd\lambda_G(x)\right)\overline{\xi(s)}\dd\lambda_N(s)
\\&=\int_Gf(x)\left(\int_N\Delta_N(s^{-1})\overline{g(xs^{-1})}\overline{\xi(s)}\dd\lambda_N(s)\right)\dd\lambda_G(x)
\\&=\int_Gf(x)\left(\int_N\Delta_N(s)\overline{g(xs)}\xi(s)\dd\lambda_N(s^{-1})\right)\dd\lambda_G(x)
\\&=\int_Gf(x)\left(\int_N\overline{g(xs)}\xi(s)\dd\lambda_N(s)\right)\dd\lambda_G(x)=\langle f,T_\xi(g)\rangle.
\end{align*}
\end{proof}

\section{\bf Harmonic Analysis on Covariant Functions of Characters of Normal Subgroups}
In this section, we consider $L^1$-spaces of covariant functions of characters of normal subgroups. We then study some of the basic properties of these classical Banach spaces of covariant functions of characters of normal subgroups. Throughout, suppose that $G$ is a locally compact group and $N$ is a closed normal subgroup of $G$. Let $\lambda_{G}$ be a left Haar measure on $G$ and $\lambda_N$ be a left Haar measure on $N$.
Also, we assume that $\lambda_{G/N}$ is the left Haar measure on the 
factor (quotient) group $G/N$ normalized with respect to Weil's formula (\ref{6}). Let $\xi\in\chi(N)$ be given. For $\psi\in \mathcal{C}_\xi(G,N)$, define the norm 
\begin{equation}\label{(1)}
\|\psi\|_{(1)}:=\left\||\psi|\right\|_{L^1(G/N,\lambda_{G/N})}=\int_{G/N}|\psi(y)|\dd\lambda_{G/N}(yN).
\end{equation}
If $\psi\in \mathcal{C}_\xi(G,N)$, the function $y\mapsto |\psi(y)|$ reduces to constant on $N$ and hence it depends only on the coset $yN$. 
Therefore, $yN\mapsto |\psi(y)|$ defines a function in $\mathcal{C}_c(G/N)$
which can be integrated with respect to $\lambda_{G/N}$.

\begin{remark}
Let $\xi=1$ be the trivial character of $N$. Suppose that  
$\psi\in \mathcal{C}_1(G,N)$ is identified with $\widetilde{\psi}\in\mathcal{C}_c(G/N)$, due to Remark \ref{J1}. Then $\|\psi\|_{(1)}=\|\widetilde{\psi}\|_{L^1(G/N)}$. 
\end{remark}

\begin{proposition}
{\it Let $G$ be a locally compact group and $N$ be a compact normal subgroup of $G$. Suppose $\xi\in\chi(N)$ and $\psi\in \mathcal{C}_\xi(G,N)$. We then have 
\begin{equation}\label{Ncompact.Lp.gen}
\|\psi\|_{(1)}=\lambda_N(N)\|\psi\|_{L^1(G)}.
\end{equation} 
}\end{proposition}
\begin{proof}
Since $N$ is compact in $G$, we get $\mathcal{C}_\xi(G,N)\subseteq\mathcal{C}_c(G)$, see Proposition 3.1(1) of \cite{AGHF.cov.com}. Also, using compactness of $N$ in $G$, each Haar measure of $N$ is finite. Then using Weil's formula (\ref{6}), for every $\psi\in \mathcal{C}_\xi(G,N)$, we get 
\begin{align*}
\|\psi\|_{L^1(G)}
&=\int_{G/N}T_N(|\psi|)(xN)\dd\lambda_{G/N}(xN)
\\&=\int_{G/N}\int_N|\psi(x)|\dd\lambda_N(s)\dd\lambda_{G/N}(xN)
\\&=\int_{G/N}\left(\int_N\dd\lambda_N(s)\right)|\psi(x)|\dd\lambda_{G/N}(xN)
=\lambda_N(N)\|\psi\|_{(1)},
\end{align*}
\end{proof}

\begin{remark}
Invoking Proposition \ref{Ncompact.Lp.gen}, if $N$ is compact in $G$ and the left Haar measure $\lambda_{G/N}$ is normalized with respect to Weil's formula and the probability measure $\lambda_N$ of $N$, we get $\|\psi\|_{(1)}=\|\psi\|_{L^1(G)}$, for every $\psi\in \mathcal{C}_\xi(G,N)$. For more details on covariant functions of characters of compact subgroups, we refer the reader to \cite{AGHF.cov.com}.
\end{remark}

Next we show that $T_\xi:(\mathcal{C}_c(G),\|\cdot\|_{L^1(G)})\to(\mathcal{C}_\xi(G,N),\|\cdot\|_{(1)})$ is a contraction.

\begin{theorem}\label{Jxi.1}
{\it Let $G$ be a locally compact group, $N$ be a closed normal subgroup of $G$, and $\xi\in\chi(N)$. The linear operator  
$T_\xi:(\mathcal{C}_c(G),\|\cdot\|_{L^1(G)})\to(\mathcal{C}_\xi(G,N),\|\cdot\|_{(1)})$ 
is a contraction. 
}\end{theorem}
\begin{proof}
Let $f\in\mathcal{C}_c(G)$ be given. Then, using Weil's formula (\ref{6}), we get  
\begin{align*}
\|T_\xi(f)\|_{(1)}
&=\int_{G/N}|T_\xi(f)(y)|\dd\lambda_{G/N}(yN)
\\&=\int_{G/N}\left|\int_Nf(ys)\overline{\xi(s)}\dd\lambda_N(s)\right|\dd\lambda_{G/N}(yN)
\\&\le\int_{G/N}\int_N\left|f(ys)\right|\dd\lambda_N(s)\dd\lambda_{G/N}(yN)
\\&=\int_{G/N}T_N(|f|)(xN)\dd\lambda_{G/N}(xN)=\|f\|_{L^1(G)}.
\end{align*}
\end{proof}

We then continue by proving the following fundamental property of $\|\cdot\|_{(1)}$.

\begin{proposition}\label{Jxi.1.1.inf}
{\it Let $G$ be a locally compact group, $N$ be a closed normal subgroup of $G$, and $\xi\in\chi(N)$. Then, for each $\psi\in \mathcal{C}_\xi(G,N)$, we have 
\[
\|\psi\|_{(1)}=\inf\left\{\|f\|_{L^1(G)}:f\in\mathcal{C}_c(G), \ T_\xi(f)=\psi\right\}.
\]
}\end{proposition}
\begin{proof}
Let $\psi\in \mathcal{C}_\xi(G,N)$ be given. Suppose $V_\psi:=\{f\in\mathcal{C}_c(G): T_\xi(f)=\psi\}$. We then define $\gamma_\psi:=\inf\left\{\|f\|_{L^1(G)}:f\in V_\psi\right\}$. Using Theorem \ref{Jxi.1}, for every $f\in V_\psi$, we have $\|\psi\|_{(1)}\le\|f\|_{L^1(G)}$. So we obtain $\|\psi\|_{(1)}\le\gamma_\psi$. We then also show that $\|\psi\|_{(1)}\ge\gamma_\psi$. To prove this, using Lemma 2.47 of \cite{FollH}, let $h\in\mathcal{C}_c(G)$ be a positive function with $T_N(h)=1$ on $\mathrm{q}(\mathrm{supp}(\psi))$.
Put $g(x):=\psi(x)h(x)$ for $x\in G$. Then $T_\xi(g)=\psi$ and hence $g\in V_\psi$. Also, we get $T_N(|g|)(xN)=|\psi(x)|$ for $x\in G$. 
Then, using Weils' formula (\ref{6}), we achieve 
\begin{align*}
\|g\|_{L^1(G)}
=\int_{G/N}T_N(|g|)(xN)\dd\lambda_{G/N}(xN)
=\|\psi\|_{(1)},
\end{align*}
implying that $\|\psi\|_{(1)}=\|g\|_{L^1(G)}\ge\gamma_\psi$.
\end{proof}

Suppose that $\mathcal{N}_\xi(G,N)$ is the kernel of the linear operator $T_\xi:\mathcal{C}_c(G)\to \mathcal{C}_\xi(G,N)$, which is the linear subspace given by 
\[
\mathcal{N}_\xi(G,N):=\{f\in\mathcal{C}_c(G):T_\xi(f)=0\}.
\]
Then, $\mathcal{N}_\xi(G,N)$ is a closed linear subspace of $(\mathcal{C}_c(G),\|\cdot\|_{L^1(G)})$. Also, by applying Proposition \ref{J.xi.RkLy}(1), we achieve    
\begin{equation}\label{Rf.xif}
\mathrm{span}\{R_kf-\Delta_N(k^{-1})\xi(k)f:k\in N,\ f\in\mathcal{C}_c(G)\}\subseteq\mathcal{N}_\xi(G,N).
\end{equation}
Let $\mathfrak{X}_\xi(G,N):=\mathcal{C}_c(G)/\mathcal{N}_\xi(G,N)$ be the quotient normed space of $\mathcal{N}_\xi(G,N)$ in $\mathcal{C}_c(G)$, that is 
\[
\mathcal{C}_c(G)/\mathcal{N}_\xi(G,N)=\left\{f+\mathcal{N}_\xi(G,N):f\in\mathcal{C}_c(G)\right\},
\]
equipped with the quotient norm given by 
\begin{equation}\label{Q.norm}
\|f+\mathcal{N}_\xi\|_{\mathfrak{X}_\xi}:=\inf\left\{\|f+g\|_{L^1(G)}:g\in\mathcal{N}_\xi(G,N)\right\}.
\end{equation}
Suppose that $\mathfrak{X}_\xi^1(G,N)$ is the Banach completion of the normed space $\mathfrak{X}_\xi(G,N)$ with respect to the quotient norm 
$\|\cdot\|_{\mathfrak{X}_\xi}$ given by (\ref{Q.norm}). 

The following result characterizes the covariant function space $\mathcal{C}_\xi(G,N)$ as a quotient space of $\mathcal{C}_c(G)$.  

\begin{theorem}\label{X1L1.M}
{\it Let $G$ be a locally compact group, $N$ be a closed normal subgroup of $G$, and $\xi\in\chi(N)$. Then, $(\mathcal{C}_\xi(G,N),\|\cdot\|_{(1)})$ is isometrically  isomorphic to the quotient space $\mathfrak{X}_\xi(G,N)$.
}\end{theorem}
\begin{proof}
Invoking Proposition 6.1 of \cite{FollH}, the linear operator $T_\xi:\mathcal{C}_c(G)\to \mathcal{C}_\xi(G,N)$ is surjective. Therefore, we deduce that the linear space $\mathcal{C}_\xi(G,N)$ is isomorphic to the quotient linear space $\mathfrak{X}_\xi(G,N)$ via the linear operator $U_\xi:\mathfrak{X}_\xi(G,N)\to \mathcal{C}_\xi(G,N)$ defined by $U_\xi(f+\mathcal{N}_\xi):=T_\xi(f)$ for every $f\in\mathcal{C}_c(G)$. The algebraic isomorphism $U_\xi$ is isometric as well, if the quotient linear space $\mathfrak{X}_\xi(G,N)$ is equipped with the quotient norm (\ref{Q.norm}). To show this, using Proposition \ref{Jxi.1.1.inf}, for every $f\in\mathcal{C}_c(G)$, we achieve  
\begin{align*}
\|U_\xi(f+\mathcal{N}_\xi)\|_{(1)}
&=\|T_\xi(f)\|_{(1)}
\\&=\inf\{\|h\|_{L^1(G)}:T_\xi(h)=T_\xi(f)\}
\\&=\inf\{\|h\|_{L^1(G)}:h-f\in\mathcal{N}_\xi\}
\\&=\inf\{\|f+g\|_{L^1(G)}:g\in\mathcal{N}_\xi\}=\|f+\mathcal{N}_\xi\|_{\mathfrak{X}_\xi}.
\end{align*}
\end{proof}

\section{\bf The Banach Space $L^1_\xi(G,N)$}

In this section, we study further properties of the $L^1$-spaces of covariant functions of characters of normal subgroups. 
Throughout, we still suppose that $G$ is a locally compact group and $N$ is a closed normal subgroup of $G$. Let $\lambda_{G}$ be a left Haar measure on $G$ and $\lambda_N$ be a left Haar measure on $N$. We also assume that $\lambda_{G/N}$ is the left Haar measure on the factor (quotient) group $G/N$ normalized with respect to Weil's formula (\ref{6}). Let $\xi\in\chi(N)$ and $L^1_\xi(G,N)$ be the Banach completion of the normed linear space $\mathcal{C}_\xi(G,N)$ with respect to $\|\cdot\|_{(1)}$ given by (\ref{(1)}). We shall use the completion norm by $\|\cdot\|_{(1)}$ as well. 

The following characterization of $L^1_\xi(G,N)$ is also a canonical consequence of Theorem \ref{X1L1.M} and construction of the Banach space $\mathfrak{X}_\xi^1(G,N)$.

\begin{proposition}\label{X1L1}
{\it Let $G$ be a locally compact group, $N$ be a closed normal subgroup of $G$, and $\xi\in\chi(N)$. Then, $\mathfrak{X}^1_\xi(G,N)$ is isometrically  isomorphic to the Banach space $L^1_\xi(G,N)$.
}\end{proposition}

\begin{remark}\label{func.real}
The abstract space $L^1_\xi(G,N)$ can be identified with a linear space of complex-valued functions on $G$, where two functions are identified when they are equal locally almost everywhere (l.a.e). Let $\mathfrak{A}\in L^1_\xi(G,N)$ be an equivalent class of $\|\cdot\|_{(1)}$-Cauchy sequences in $\mathcal{C}_\xi(G,N)$. Suppose $(\psi_n)_{n=1}^\infty$ is a representative for $\mathfrak{A}$. So $(\psi_n)_{n=1}^\infty$ is a $\|\cdot\|_{(1)}$-Cauchy sequence in $\mathcal{C}_\xi(G,N)$ and $\|\mathfrak{A}\|_{(1)}=\lim_n\|\psi_n\|_{(1)}$. Then choose a subsequence $(\psi_{\epsilon(k)})$ 
such that $\|\psi_{\epsilon(k+1)}-\psi_{\epsilon(k)}\|_{(1)}\le1/2^k$, for $k\in \mathbb{N}$. This implies that 
$\sum_{k=1}^\infty\|\psi_{\epsilon(k+1)}-\psi_{\epsilon(k)}\|_{(1)}<\infty$.
Hence, $\sum_{k=1}^\infty\|\Psi_k\|_{L^1(G/N)}<\infty$ with $\Psi_k:=|\psi_{\epsilon(k+1)}-\psi_{\epsilon(k)}|$ for $k\in\mathbb{N}$. Since $L^1(G/N)$ is a Banach space, there exists $\Psi\in L^1(G/N)$ with $\Psi=\sum_{k=1}^\infty\Psi_k$ in $L^1(G/N)$, where the sum converges pointwise almost everywhere as well. Thus, 
\[
\sum_{k=1}^\infty|\psi_{\epsilon(k+1)}(x)-\psi_{\epsilon(k)}(x)|=\sum_{k=1}^\infty\Psi_k(xN)<\infty,
\]
for $xN\in G/N$ except those in a subset set $E$ of $G/N$ with $\lambda_{G/N}(E)=0$. So, the complex series $\sum_{k=1}^\infty\psi_{\epsilon(k+1)}(x)-\psi_{\epsilon(k)}(x),$ converges to a complex number, denoted by $\psi_0(x)$, for $x\in G$ with $xN\not\in E$. For $x\in G$ with $xN\notin E$, let $\psi(x):=\psi_{\epsilon(1)}(x)+\psi_0(x)$. Further, $\psi(x)$ is independent of the choice of the Cauchy sequence $(\psi_n)$ as a representative of the class $\mathfrak{A}$, since locally null sets in $G$ 
are precisely the images under the canonical projection of the locally null sets in $G/N$, Theorem 3.3.28 of \cite{50}. The class $\mathfrak{A}$ can be uniquely determined with the locally almost everywhere defined function $x\mapsto\psi(x)$, as a complex valued function on $G$. 
Hence, we get 
\[
\lim_{n}\psi_{\epsilon(n)}(x)=\lim_n\left(\psi_{\epsilon(1)}(x)+\sum_{k=1}^{n-1}\psi_{\epsilon(k+1)}(x)-\psi_{\epsilon(k)}(x)\right)=\psi(x).
\]
Under this identification, for each $n$, we have 
\[
|\psi_{\epsilon(n)}(x)|\le|\psi_{\epsilon(1)}(x)|+\sum_{k=1}^{n-1}\Psi_k(x)\le |\psi_{\epsilon(1)}(x)|+\Psi(x).
\]
Then Lebesgue's Dominated Convergence Theorem, implies that the function $xN\mapsto |\psi(x)|$ belongs to $L^1(G/N)$ and 
\begin{equation}
\|\mathfrak{A}\|_{(1)}=\lim_{n}\|\psi_{\epsilon(n)}\|_{(1)}=\int_{G/N}|\psi(x)|\dd\lambda_{G/N}(xN).
\end{equation}
\end{remark}

Invoking the structure of the Banach space  $L^1_\xi(G,N)$ and Theorem \ref{Jxi.1}, we conclude that the linear operator  
$T_\xi:(\mathcal{C}_c(G),\|\cdot\|_{L^1(G)})\to(\mathcal{C}_\xi(G,N),\|\cdot\|_{(1)})$ has a unique extension to a contraction from $L^1(G)$ into $L^1_\xi(G,N)$, which we still denote this unique bounded linear operator by $T_\xi:L^1(G)\to L^1_\xi(G,N)$.

We then demonstrate the following explicit formula for $T_\xi:L^1(G)\to L^1_\xi(G,N)$. 

\begin{theorem}\label{J.xi.L1}
Let $G$ be a locally compact group, $N$ be a closed normal subgroup of $G$, and $\xi\in\chi(N)$. The extended linear operator $T_\xi:L^1(G)\to L^1_\xi(G,N)$ is given by $f\mapsto T_\xi(f)$, where  
\begin{equation}\label{Txi.L1.formula}
T_\xi(f)(x)=\int_Nf(xs)\overline{\xi(s)}\dd\lambda_N(s),\ {\rm for\ \ l.a.e.}\ x\in G.
\end{equation}
\end{theorem}
\begin{proof}
Suppose that $f\in L^1(G)$ is given. For $x\in G$, let $f^x:N\to\mathbb{C}$ be given by $f^x(s):=f(xs)$ for $s\in N$. Invoking Theorem 3.4.6 of \cite{50}, there exists a $\lambda_{G/N}$-null set $A$ in $G/N$ such that $f^x\in L^1(N)$ for every $x\in G$ with $xN\not\in A$. Since any $\lambda_{G/N}$-null set is a locally $\lambda_{G/N}$-null set, using Theorem 3.3.28 of \cite{50}, we conclude that $\mathrm{q}^{-1}(A)$ is  locally $\lambda_G$-null in $G$.
So $f^x\in L^1(N)$ for every $x\in G$ with $x\not\in\mathrm{q}^{-1}(A)$.
This implies that the right hand side of (\ref{Txi.L1.formula}) is well-defined as a function on $G$, for l.a.e $x\in G$. 
Let $(f_n)\subset\mathcal{C}_c(G)$ with $\|f_n-f\|_{L^1(G)}<2^{-(n+1)}$ for every $n\in\mathbb{N}$. Then $\|f_{n+1}-f_n\|_{L^1(G)}< 2^{-n}$ for $n\in \mathbb{N}$. We also have $T_\xi(f)=\lim_nT_\xi(f_n)$ in $L^1_\xi(G,N)$. Let $\psi_n=T_\xi(f_n)$ for $n\in\mathbb{N}$. Invoking Theorem \ref{Jxi.1}, we get 
$\|\psi_{n+1}-\psi_n\|_{(1)}<2^{-n}$ for $n\in\mathbb{N}$. 
This implies that $\lim_n\psi_{n}(x)=T_\xi(f)(x)$ for l.c.a $x\in G$, according to Remark \ref{func.real}. 
Using Weil's formula (\ref{6}), for $n\in\mathbb{N}$, we obtain  
\[
\int_{G/N}\int_N|f_{n}(xs)-f(xs)|\dd\lambda_{N}(s)\dd\lambda_{G/N}(xN)=\|f_{n}-f\|_{L^1(G)},
\]
which implies that 
\[
\lim_n\int_N|f_{n}(xs)-f(xs)|\dd\lambda_{N}(s)=0,
\]
for a.e. $xN\in G/N$. Since for $n\in\mathbb{N}$ and $x\in G$, we have  
\begin{align*}
\left|\int_Nf_{n}(xs)\overline{\xi(s)}\dd\lambda_{N}(s)-\int_Nf(xs)\overline{\xi(s)}\dd\lambda_{N}(s)\right|
&\le\int_N|f_{n}(xs)-f(xs)|\dd\lambda_{N}(s),
\end{align*}
we conclude that 
\begin{equation}\label{limN.main}
\lim_n\int_Nf_{n}(xs)\overline{\xi(s)}\dd\lambda_{N}(s)=\int_Nf(xs)\overline{\xi(s)}\dd\lambda_{N}(s),
\end{equation}
for a.e. $xN\in G/N$.
Using (\ref{limN.main}), for l.c.a $x\in G$, we achieve 
\begin{align*}
T_\xi(f)(x)&=\lim_nT_\xi(f_{n})(x)=\lim_n\int_Nf_{n}(xs)\overline{\xi(s)}\dd\lambda_{N}(s)
=\int_Nf(xs)\overline{\xi(s)}\dd\lambda_{N}(s).
\end{align*}
\end{proof}

Assume that $\mathcal{N}_\xi^1(G,N)$ is the kernel of the linear map $T_\xi:L^1(G)\to L^1_\xi(G,N)$ in $L^1(G)$, which is the linear subspace 
\[
\mathcal{N}_\xi^1(G,N):=\left\{f\in L^1(G):T_\xi(f)=0\ {\rm in}\ L^1_\xi(G,N)\right\}.
\]
We then show that $\mathcal{N}_\xi^1(G,N)$ is the closure of $\mathcal{N}_\xi(G,N)$ in $L^1(G)$.

\begin{proposition}\label{N1}
{\it Let $G$ be a locally compact group, $N$ be a closed normal subgroup of $G$, and $\xi\in\chi(N)$. Then, $\mathcal{N}_\xi^1(G,N)$ is the closure of $\mathcal{N}_\xi(G,N)$ in $L^1(G)$.
}\end{proposition}
\begin{proof}
Suppose that $\mathcal{X}$ denotes the closure of $\mathcal{N}_\xi(G,N)$ in $L^1(G)$. Let $f\in\mathcal{X}$ be arbitrary. Then, $f\in L^1(G)$ and 
$\lim_nf_n=f$ for some sequence $(f_n)\subset\mathcal{N}_\xi(G,N)$. 
Then continuity of the linear operator $T_\xi:L^1(G)\to L^1_\xi(G,N)$ guarantees that $T_\xi(f)=0$ in $L^1_\xi(G,N)$. Therefore, $f\in \mathcal{N}_\xi^1(G,N)$. Since $f$ was given, we get $\mathcal{X}\subseteq\mathcal{N}_\xi^1(G,N)$. Conversely, suppose $f\in L^1(G)$ with $T_\xi(f)=0$ in $L^1_\xi(G,N)$.
Let $\varepsilon>0$ be given. Pick $h\in\mathcal{C}_c(G)$ such that $\|f-h\|_{L^1(G)}<\varepsilon/2$ and define $\phi:=T_\xi(h)$. Let $u\in\mathcal{C}_c(G)$ be a non-negative function with $T_N(u)=1$ on $\mathrm{q}(\mathrm{supp}(\phi))$, according to Lemma 2.47 of \cite{FollH}. We then consider the function $g:G\to\mathbb{C}$ defined by $g(x):=h(x)-u(x)\phi(x)$, for every $x\in G$. Then, $g\in\mathcal{C}_c(G)$ and $T_\xi(g)=0$. Because, for every $x\in G$, we obtain  
\[
T_\xi(g)(x)=\phi(x)-\int_Nu(xs)\phi(xs)\overline{\xi(s)}d\lambda_N(s)=\phi(x)-\phi(x)T_N(u)(xN)=0.
\]
Hence, we conclude that $g\in\mathcal{N}_\xi(G,N)$. Then, using Weil's formula (\ref{6}), we achieve  
\begin{align*}
\|h-g\|_{L^1(G)}&=\int_Gu(x)|\phi(x)|\dd\lambda_G(x)=\int_{G/N}T_N(u|\phi|)\dd\lambda_G(x)
\\&=\int_{G/N}T_N(u)(xN)|\phi(x)|\dd\lambda_G(x)=\int_{G/N}|\phi(x)|\dd\lambda_G(x)
\\&=\|T_\xi(h)\|_{(1)}
\le\|T_\xi(h-f)\|_{(1)}+\|T_\xi(f)\|_{(1)}\le\|h-f\|_{L^1(G)}<\frac{\varepsilon}{2}.
\end{align*}
Therefore, we have
\[
\|f-g\|_{L^1(G)}\le\|f-h\|_{L^1(G)}+\|h-g\|_{L^1(G)}<\varepsilon,
\]
which implies $f\in\mathcal{X}$. Since $f\in\mathcal{N}_\xi^1(G,N)$ was arbitrary, we get $\mathcal{N}_\xi^1(G,N)\subseteq\mathcal{X}$.
\end{proof}

Next we have the following characterization of $L^1_\xi(G,N)$ as a quotient subspace of $L^1(G)$. 

\begin{theorem}\label{L1.main}
{\it Let $G$ be a locally compact group, $N$ be a closed normal subgroup of $G$, and $\xi\in\chi(N)$. The Banach space $L_\xi^1(G,N)$ is isometrically isomorphic to the quotient Banach space $L^1(G)/\mathcal{N}_\xi^1(G,N)$.
}\end{theorem}
\begin{proof}
Invoking Proposition \ref{N1} and also Lemma 3.4.4 of \cite{50}, one can conclude that the Banach space $\mathfrak{X}_\xi^1(G,N)$ is isometrically isomorphic to the quotient Banach space $L^1(G)/\mathcal{N}_\xi^1(G,N)$. Then, Proposition \ref{X1L1} guarantees that the Banach space $L_\xi^1(G,N)$ is isometrically isomorphic to the quotient Banach space $L^1(G)/\mathcal{N}_\xi^1(G,N)$.
\end{proof}

\begin{corollary}
{\it Let $G$ be a locally compact group and $N$ be a closed normal subgroup of $G$. Suppose $\xi\in\chi(N)$ and $\psi\in L^1_\xi(G,N)$. We then have 
\[
\|\psi\|_{(1)}=\inf\left\{\|f\|_{L^1(G)}:f\in L^1(G), \ T_\xi(f)=\psi\right\}.
\]
}\end{corollary}

\begin{corollary}
{\it Let $G$ be a locally compact group, $N$ be a closed normal subgroup of $G$, and $\xi\in\chi(N)$. The extended contraction $T_\xi$ maps $L^1(G)$ onto $L^1_\xi(G,N)$.
}\end{corollary}

Let $L^\infty(G)$ be the Banach space of all locally $\lambda_G$-measurable functions $f:G\to\mathbb{C}$ which are bounded except on a locally $\lambda_G$-null set, modulo functions which are zero locally a.e. on $G$, equipped with the norm
\[
\|f\|_\infty:=\inf\{t: |f(x)|\le t\ \ {\rm l.a.e.}\ \ x\in G\}.
\]
Then $\mathcal{C}_\xi(G,N)\subset L^\infty(G)$. Suppose that $L^\infty_\xi(G,N)$ is the closed subspace of $L^\infty(G)$ defined by  
\[
L^\infty_\xi(G,N):=\left\{\psi\in L^\infty(G):R_k\psi=\xi(k)\psi,\ {\rm for}\ k\in N\right\}.
\]
Then $\mathcal{C}_\xi(G,N)\subseteq L^\infty_\xi(G,N)$ and hence the closure of $\mathcal{C}_\xi(G,N)$ in $L^\infty(G)$ is included in $L^\infty(G,N)$.

We conclude the paper by the following characterization for the dual space $L^1_\xi(G,N)^*$.

\begin{theorem}
{\it Let $G$ be a locally compact group, $N$ be a closed normal subgroup of $G$, and $\xi\in\chi(N)$. Then, $L^1_\xi(G,N)^*$ is isometrically isomorphic to $L^\infty_\xi(G,N)$. 
}\end{theorem}
\begin{proof}
Let $g\mapsto \Lambda_g$ be the antilinear isometric isomorphism identification of $L^\infty(G)$ as $L^1(G)^*$, where $\Lambda_g:L^1(G)\to\mathbb{C}$ is given by $\Lambda_g(f):=\langle f,g\rangle$ for every $f\in L^1(G)$. Using Theorem \ref{L1.main}, the dual space 
$L^1_\xi(G,N)^*$ is isometrically isomorphic to $\mathcal{N}_\xi^1(G,N)^\perp$, where 
\[
\mathcal{N}_\xi^1(G,N)^\perp=\left\{g\in L^\infty(G):\Lambda_g(f)=0,\ \ {\rm for\ all}\ f\in \mathcal{N}_\xi^1(G,N)\right\}.
\]
We then show that $\mathcal{N}_\xi^1(G,N)^\perp=L^\infty_\xi(G,N)$. To this end, suppose that $g\in\mathcal{N}_\xi^1(G,N)^\perp$ is arbitrary. Then, $g\in L^\infty(G)$ and $\Lambda_g(f)=0$ for every $f\in \mathcal{N}_\xi^1(G,N)$. Let $k\in N$ be given. Then, using (\ref{Rf.xif}) and Proposition \ref{N1}, for every $f\in\mathcal{C}_c(G)$ we get $R_{k^{-1}}f-\Delta_N(k)\overline{\xi(k)}f\in\mathcal{N}_\xi^1(G,N)$. So $\Lambda_g(R_{k^{-1}}f)=\Delta_N(k)\overline{\xi(k)}\Lambda_g(f)$ and hence 
$\Delta_N(k^{-1})\Lambda_g(R_{k^{-1}}f)=\overline{\xi(k)}\Lambda_g(f)$, for every $f\in\mathcal{C}_c(G)$. Since $N$ is normal, we have $\Delta_G|_N=\Delta_N$. Therefore, for every $f\in\mathcal{C}_c(G)$, we achieve 
\begin{align*}
\Lambda_{R_kg}(f)&=\int_Gf(x)\overline{g(xk)}\dd\lambda_G(x)
\\&=\int_Gf(xk^{-1})\overline{g(x)}\dd\lambda_G(xk^{-1})
\\&=\Delta_G(k^{-1})\int_Gf(xk^{-1})\overline{g(x)}\dd\lambda_G(x)
\\&=\Delta_N(k^{-1})\int_Gf(xk^{-1})\overline{g(x)}\dd\lambda_G(x)
\\&=\Delta_N(k^{-1})\Lambda_g(R_{k^{-1}}f)=\overline{\xi(k)}\Lambda_g(f)=\Lambda_{\xi(k)g}(f).
\end{align*} 
Since $f\in\mathcal{C}_c(G)$ is given, we obtain $R_kg=\xi(k)g$ in $L^\infty(G)$. Because $k\in N$ was also arbitrary, we get $g\in L^\infty_\xi(G,N)$. Conversely, let $g\in L^\infty_\xi(G,N)$ be given. We then claim that $\Lambda_g(f)=0$ for every  $f\in \mathcal{N}_\xi^1(G,N)$. To this end, suppose that $f\in \mathcal{N}_\xi^1(G,N)$ is given. Then
\begin{align*}
\Lambda_g(f)&=\int_G f(x)\overline{g(x)}\dd\lambda_G(x)
\\&=\int_{G/N}\left(\int_Nf(xs)\overline{g(xs)}\dd\lambda_N(s)\right)\dd\lambda_{G/N}(xN)
\\&=\int_{G/N}\left(\int_Nf(xs)\overline{\xi(s)}\dd\lambda_N(s)\right)\overline{g(x})\dd\lambda_{G/N}(xN)
\\&=\int_{G/N}T_\xi(f)(x)\overline{g(x})\dd\lambda_{G/N}(xN)=0,
\end{align*}
implying that $g\in \mathcal{N}_\xi^1(G,N)^\perp$.
\end{proof}

{\bf Acknowledgement.}
This project has received funding from the European Union’s Horizon 2020 research and innovation programme under the Marie Sklodowska-Curie grant agreement No. 794305. The author gratefully acknowledges the supporting agency. The findings and opinions expressed here are only those of the author, and not of the funding agency.\\
The author would like to express his deepest gratitude to Vladimir V. Kisil for suggesting the problem that motivated the results in this article, 
stimulating discussions and pointing out various references. 

\bibliographystyle{amsplain}

\end{document}